\documentclass[12pt]{article}
\usepackage{lineno,hyperref}
\modulolinenumbers[5]
\usepackage{color}

\usepackage{graphics,amsmath,amssymb}
\usepackage{amsthm}
\usepackage{amsfonts}
\usepackage{latexsym}
\usepackage{epsf}

\newcommand{\sstirling}[2]{\genfrac\{\}{0pt}{}{#1}{#2}}
\newcommand{\rbstirling}[2]{\genfrac\langle\rangle{0pt}{}{#1}{#2}}

\def\a{{\bf a}}
\def\s{\mathfrak{S}}
\def\b{\mathfrak{B}}

\newtheorem{theorem}{Theorem}

\newtheorem{corollary}[theorem]{Corollary}

\begin{document}

\begin{center}
\vskip 1cm{\LARGE\bf Some Generalizations of Spivey's Bell Number Formula}
\vskip 1cm
\large 
{\bf Mahid M. Mangontarum}\\
Department of Mathematics\\
Mindanao State University -- Main Campus\\
Marawi City 9700\\
Philippines\\
\href{mailto:mmangontarum@yahoo.com}{\tt mmangontarum@yahoo.com}\\
\href{mailto:mangontarum.mahid@msumain.edu.ph}{\tt mangontarum.mahid@msumain.edu.ph}\\
\vskip 1cm
{\bf Amerah M. Dibagulun}\\
Department of Mathematics\\
Mindanao State University -- Main Campus\\
Marawi City 9700\\
Philippines\\
\href{mailto:a.dibagulun@yahoo.com}{\tt a.dibagulun@yahoo.com}\\
\end{center}

\begin{abstract}
In this paper, a generalized recurrence relation for the $r$-Whitney numbers of the second kind is derived using as framework the operators $X$ and $D$ satisfying the commutation relation $DX-XD=1$. This recurrence relation is shown to be a generalization of the well-known Spivey's Bell number. Moreover, several other identities generalizing Spivey's Bell number formula are obtained.

\medskip


\medskip


\end{abstract}

\section{Introduction}

The \emph{Stirling numbers of the second kind}, denoted by $\sstirling{n}{j}$, are known to count the number of partitions of a set with $n$ elements into $k$ non-empty subsets. These numbers are also known to be coefficients in the expansion of
\begin{equation}
t^n=\sum_{k=0}^n\sstirling{n}{k}(t)_k,
\end{equation}
where $(t)_k=t(t-1)(t-2)\cdots(t-k+1)$ (see \cite{Comt}). The sum of the Stirling numbers of the second kind, called \emph{Bell numbers} and denoted by $B_n$, count the total number of partitions of a set with $n$ elements. They are given by
\begin{equation}
B_n=\sum_{j=0}^{n}\sstirling{n}{j}\label{Belldef}
\end{equation}
and are known to satisfy the recurrence relation
\begin{equation}
B_{n+1}=\sum_{k=0}^n\binom{n}{k}B_k.\label{Bellrec}
\end{equation}
Using the combinatorial meanings of the numbers $\sstirling{n}{k}$ and $B_n$, Spivey \cite{Spivey} was able to obtain a generalized recurrence formula for $B_n$ which unifies Equations \eqref{Belldef} and \eqref{Bellrec}. The said formula, popularly known as ``Spivey's Bell number formula'', is given by
\begin{equation}
B_{n+m}=\sum_{k=0}^n\sum_{j=0}^mj^{n-k}\sstirling{m}{j}\binom{n}{k}B_k.\label{spivey}
\end{equation}
After its formulation, several mathematicians became interested in studying alternative proofs and some extensions of Spivey's Bell number formula. In particular, Gould and Quaintance \cite{Gould} proved \eqref{spivey} using generating functions while Belbachir and Mihoubi \cite{Belbachir} made use of a decomposition of the \emph{Bell polynomials} into a certain basis. Mez\H{o} \cite{Mez} obtained a generalization of Spivey's Bell number formula in terms of the \emph{$r$-Bell numbers} defined by
\begin{equation}
B_{n,r}=\sum_{k=0}^n\sstirling{n}{k}_r,
\end{equation} 
where $\sstirling{n}{k}_r$ denote the \emph{$r$-Stirling numbers of the second kind} (see \cite{Broder}). The said generalization is given by
\begin{equation}
B_{n+m,r}=\sum_{k=0}^n\sum_{j=0}^m(j+r)^{n-k}\sstirling{m}{j}_r\binom{n}{k}B_k
\end{equation}
and was proved using the combinatorial interpretation of the $r$-Stirling numbers of the second kind and the $r$-Bell numbers. On the other hand, Katriel \cite{Katriel} obtained the following \emph{$q$-analogue} of Spivey's Bell number formula for the case of the \emph{$q$-Bell numbers} $B_n(q)$:
\begin{equation}
B_{n+m}(q)=\sum_{k=0}^n\sum_{j=0}^m\sstirling{m}{j}_q\binom{n}{k}[j]_q^{n-k}q^{jk}B_k(q).\label{katrielresult}
\end{equation}
Here, $[j]_q=\frac{q^j-1}{q-1}$ is the \emph{$q$-integer} and $\sstirling{m}{j}_q$ denote the \emph{$q$-Stirling numbers of the second kind} defined by
\begin{equation}
(XD)^n=\sum_{=0}^n\sstirling{n}{k}_qX^kD^k,
\end{equation}
where $X$ and $D$ are the operators satisfying the commutation relation
\begin{equation}
[D,X]_q:=DX-qXD=1\label{qcommrel}
\end{equation}
(see \cite[Section\ 2]{Katriel}). Other related studies are due to Xu \cite{Xu}, Mansour et al. \cite{Mansour2} and Corcino et al. \cite{CorcinoCelesteGonzales}. Their results are briefly discussed in Section \ref{sec4}.

Following a method analogous to the work of Katriel \cite{Katriel}, the authors obtained an identity that includes \eqref{spivey} as a special case. This involves the \emph{$r$-Whitney numbers of the second kind}, a natural generalization of the Stirling numbers of the second kind, which is discussed in Section \ref{sec2}.

\section{The $r$-Whitney Numbers of the Second Kind}\label{sec2}

Mez\H{o} \cite{Mez1} defined the $r$-Whitney numbers of the second kind, denoted by $W_{m,r}(n,k)$, as coefficients in the expansion of
\begin{equation}
(mt+r)^n=\sum_{k=0}^n m^kW_{m,r}(n,k)(t)_k,\label{rWhitney}
\end{equation}
for any real numbers $m$ and $r$. Notice that when $m=1$ and $r=0$, the numbers $W_{m,r}(n,k)$ reduce back to the Stirling numbers of the second kind; i.e.,
\begin{equation}
W_{1,0}(n,k)=\sstirling{n}{k}.\label{d2}
\end{equation}
Furthermore, other generalizations and extensions of the Stirling numbers of the second kind can be obtained from \eqref{rWhitney} by assigning suitable values for the parameters $m$ and $r$. To be precise, the $r$-Stirling numbers of the second kind $\sstirling{n}{k}_r$ by Broder \cite{Broder} with horizontal generating function given by
\begin{equation}
(t+r)^n=\sum_{k=0}^n\sstirling{n+r}{k+r}_r(t)_k
\end{equation}
is the case when $m=1$ in \eqref{rWhitney}; i.e.,
\begin{equation}
W_{1,r}(n,k)=\sstirling{n+r}{k+r}_r.\label{x1}
\end{equation}
The \emph{Whitney numbers of the second kind $W_m(n,k)$ of Dowling lattices} by Benoumhani \cite{Benoumhani} defined by
\begin{equation}
(mt+1)^n=\sum_{k=0}^nm^kW_m(n,k)(t)_k
\end{equation}
is the case when $r=1$ in \eqref{rWhitney}; i.e.,
\begin{equation}
W_{m,1}(n,k)=W_m(n,k).\label{x2}
\end{equation}	
The \emph{non-central Stirling numbers of the second kind $S_{\alpha}(n,k)$} by Koutras \cite{Koutras} defined by
\begin{equation}
(t-\alpha)^n=\sum_{k=0}^nS_\alpha(n,k)(t)_k
\end{equation}
is the case when $m=1$ and $r=-\alpha$ in \eqref{rWhitney}; i.e.,
\begin{equation}
W_{1,-\alpha}(n,k)=S_{\alpha}(n,k).\label{x3}
\end{equation}
The \emph{translated Whitney numbers of the second kind $\widetilde{W}_{(\alpha)}(n,k)$} by Belbachir and Bousbaa \cite{Bel} (thoroughly discussed in \cite{Mangontarum3,Mah1}) with horizontal generating function given by
\begin{equation}
t^n=\sum_{k=0}^n\widetilde{W}_{(\alpha)}(n,k)(t|\alpha)_k,
\end{equation}
where $(t|\alpha)_k=\prod_{i=0}^{k-1}(t-i\alpha)$, is the case when $r=0$ and $m=\alpha$ in \eqref{rWhitney}; i.e.,
\begin{equation}
W_{\alpha,0}(n,k)=\widetilde{W}_{(\alpha)}(n,k),\label{x4}
\end{equation}
Finally, the \emph{$(r,\beta)$-Stirling Numbers $\rbstirling{n}{k}_{r,\beta}$} \cite{Corcino}, the \emph{Ruci\'{n}ski-Voigt numbers $S^{n}_{k}(\a)$} \cite{Rucin} and the \emph{non-central Whitney numbers of the second kind $\widetilde{W}_{m,a}(n,k)$} \cite{Mah3} defined by
\begin{equation}
t^n=\sum_{k=0}^n\binom{\frac{t-r}{\beta}}{k}\beta^kk!\rbstirling{n}{k}_{r,\beta},
\end{equation}
\begin{equation}
t^n=\sum_{k=0}^nS^{n}_{k}(\a)P^{\a}_{k}(x),
\end{equation}
where $\a=(a,a+r,a+2r,a+3r,\ldots)$ and $P^{\a}_{k}(x)=\prod_{i=0}^{k-1}(t-a+ir)$, and 
\begin{equation}
(mt-a)^n=\sum_{k=0}^nm^k\widetilde{W}_{m,a}(n,k)=W_{m,-a}(n,k)(t)_k,
\end{equation}
respectively, can also be shown to be equivalent to the $r$-Whitney numbers of the second kind by carefully comparing their defining relations with \eqref{rWhitney}. By doing so, it can be easily seen that
\begin{equation}
\rbstirling{n}{k}_{r,\beta}=W_{\beta,r}(n,k),
\end{equation}
\begin{equation}
S^{n}_{k}(\a)=W_{r,a}(n,k),
\end{equation}
and
\begin{equation}
\widetilde{W}_{m,a}(n,k)=W_{m,-a}(n,k).
\end{equation}

Now, consider the classical operators $X$ and $D$ defined by
\begin{equation}
Xf(x)=xf(x) \label{d4}
\end{equation}
and 
\begin{equation}
Df(x)=\frac{d}{dx}f(x), 
\end{equation}
respectively, which are known to satisfy the commutation relation
\begin{equation}
[D,X]:=DX-XD=1.\label{commrel}
\end{equation}
Since $Dx^n=nx^{n-1}$ and $Df(x)=0$ when $f(x)=c$ ($c$ is a constant), then by \eqref{commrel},
\begin{equation}
[D,X^k]:=DX^k-X^kD=kX^{k-1}.
\end{equation}
Consequently, we get
\begin{equation}
DX^k=X^kD+kX^{k-1}.\label{d1}
\end{equation}
We are now ready to state the following theorem which expresses the $r$-Whitney numbers of the second kind in terms of $X$ and $D$:
\begin{theorem}
The $r$-Whitney numbers of the second kind satisfy the following relation:
\begin{equation}
(mXD+r)^n=\sum_{k=0}^nm^kW_{m,r}(n,k)X^kD^k.\label{altrWhitney}
\end{equation}
\end{theorem}
\begin{proof}
We proceed by induction on $n$. Clearly, \eqref{altrWhitney} holds when $n=0$. Now, assume that \eqref{altrWhitney} holds for $n>0$. Using the recurrence relation (cf. \cite{Mez1}) given by
\begin{equation}
W_{m,r}(n,k)=W_{m,r}(n-1,k-1)+(mk+r)W_{m,r}(n-1,k)\label{recrWhitney}
\end{equation}
gives
\begin{eqnarray*}
\sum_{k=0}^{n+1}m^kW_{m,r}(n+1,k)X^kD^k&=&\sum_{k=0}^{n+1}m^k\left\{W_{m,r}(n,k-1)+(mk+r)W_{m,r}(n,k)\right\}X^kD^k\\
&=&\sum_{k=0}^{n+1}m^kW_{m,r}(n,k-1)X^kD^k+\sum_{k=0}^{n+1}m^{k+1}kW_{m,r}(n,k)X^kD^k\\
& &\ +\sum_{k=0}^{n+1}rm^kW_{m,r}(n,k)X^kD^k\\
&=&\sum_{k=0}^nm^{k+1}W_{m,r}(n,k)X(X^kD+kX^{k-1})D^k\\
& &\ +\sum_{k=0}^nrm^kW_{m,r}(n,k)X^kD^k.
\end{eqnarray*}
Using \eqref{d1}, and by the inductive hypothesis,
\begin{eqnarray*}
\sum_{k=0}^{n+1}m^kW_{m,r}(n+1,k)X^kD^k&=&\sum_{k=0}^nm^{k+1}W_{m,r}(n,k)X(DX^k)D^k+\sum_{k=0}^nrm^kW_{m,r}(n,k)X^kD^k\\
&=&(mXD+r)\sum_{k=0}^nm^kW_{m,r}(n,k)X^kD^k\\
&=&(mXD+r)(mXD+r)^n\\
&=&(mXD+r)^{n+1}.
\end{eqnarray*}
This means that \eqref{altrWhitney} holds for $n+1>0$. This completes the proof.
\end{proof}

\section{A Generalization of Spivey's Bell Number Formula}\label{sec3}

Before proceeding, recall that Cheon and Jung \cite{Cheon} defined the \emph{$r$-Dowling polynomials}, denoted by $D_{m,r}(n;x)$, by
\begin{equation}
D_{m,r}(n;x)=\sum_{k=0}^nW_{m,r}(n,k)x^k,\label{rdowlingdef}
\end{equation}
where $D_{m,r}(n):=D_{m,r}(n;1)$ denote the \emph{$r$-Dowling numbers}.

The following theorem is the main result of this paper:
\begin{theorem}\label{maintheorem}
For non-negative integers $n$ and $\ell$, and real numbers $m$ and $r$, the $r$-Dowling polynomials satisfy the following recurrence relation:
\begin{equation}
D_{m,r}(n+\ell;x)=\sum_{j=0}^{\ell}\sum_{k=0}^nW_{m,r}(\ell,j)\binom{n}{k}(mj)^{n-k}D_{m,r}(k;x)x^j.\label{mainresult}
\end{equation}
Consequently, the $r$-Dowling numbers satisfy the following recurrence relation:
\begin{equation}
D_{m,r}(n+\ell)=\sum_{j=0}^{\ell}\sum_{k=0}^nW_{m,r}(\ell,j)\binom{n}{k}(mj)^{n-k}D_{m,r}(k).\label{mainresult2}
\end{equation}
\end{theorem}

\begin{proof}
Applying both sides of \eqref{altrWhitney} to the exponential function $e^x$ while keeping in mind that $De^x=e^x$ yields
\begin{eqnarray*}
\frac{1}{e^x}(mXD+r)^ne^x&=&\frac{1}{e^x}\sum_{k=0}^nm^kW_{m,r}(n,k)X^kD^ke^x\\
&=&\frac{1}{e^x}\sum_{k=0}^nW_{m,r}(n,k)(mX)^ke^x.
\end{eqnarray*}
Hence, by \eqref{rdowlingdef},
\begin{equation}
\frac{1}{e^x}(mXD+r)^ne^x=\frac{1}{e^x}D_{m,r}(n;mX)e^x.\label{r2}
\end{equation}
It is important to note that the expression $D_{m,r}(n;mX)$ does not strictly refer to the $r$-Dowling polynomials but to a specialization when $x$ is replaced with $mX$. Now, \eqref{d1} can be further expressed as
\begin{equation}
(XD)X^k=X^k\big(k+XD\big).
\end{equation}
Multiplying both sides of this identity by $m$ and then adding $rX^k$, gives
\begin{equation}
\big(mXD+r\big)X^k=X^k\big(mk+r+mXD\big).\label{r1}
\end{equation}
Combining this with \eqref{altrWhitney} and then applying the binomial theorem yields
\begin{eqnarray*}
(mXD+r)^{n+\ell}&=&(mXD+r)^n\sum_{j=0}^{\ell}m^jW_{m,r}(\ell,j)X^jD^j\\
&=&\sum_{j=0}^{\ell}m^jW_{m,r}(\ell,j)X^j\big(mj+r+mXD\big)^nD^j\\
&=&\sum_{j=0}^{\ell}\sum_{k=0}^nm^jW_{m,r}(\ell,j)\binom{n}{k}X^j(mj)^{n-k}(mXD+r)^kD^j.
\end{eqnarray*}
By \eqref{r2}, the left-hand side becomes
\begin{equation}
\frac{1}{e^x}(mXD+r)^{n+\ell}e^x=\frac{1}{e^x}D_{m,r}(n+\ell;mX)e^x.\label{d3}
\end{equation}
The right-hand side, when applied to $e^x$, becomes\\
\noindent$\displaystyle\frac{1}{e^x}\sum_{j=0}^{\ell}\sum_{k=0}^nm^jW_{m,r}(\ell,j)\binom{n}{k}X^j(mj)^{n-k}(mXD+r)^kD^je^x$
\begin{eqnarray*}
&=&\frac{1}{e^x}\sum_{j=0}^{\ell}\sum_{k=0}^nm^jW_{m,r}(\ell,j)\binom{n}{k}X^j(mj)^{n-k}\sum_{i=0}^km^iW_{m,r}(k,i)X^iD^je^x\\
&=&\frac{1}{e^x}\sum_{j=0}^{\ell}\sum_{k=0}^nm^jW_{m,r}(\ell,j)\binom{n}{k}X^j(mj)^{n-k}D_{m,r}(k;mX)e^x.
\end{eqnarray*}
Since $\frac{1}{e^x}e^x=1$, then combining the last equality with \eqref{d3} and then using \eqref{d4} gives
\begin{equation}
D_{m,r}(n+\ell;mx)=\sum_{j=0}^{\ell}\sum_{k=0}^nW_{m,r}(\ell,j)\binom{n}{k}(mx)^j(mj)^{n-k}D_{m,r}(k;mx).
\end{equation}
The desired result in \eqref{mainresult} is obtained when $mx$ is replaced with $x$. Finally, \eqref{mainresult2} is the case when $x=1$ in \eqref{mainresult}.
\end{proof}

When $m=1$, $r=0$ and $x=1$ in \eqref{mainresult}, we get \eqref{spivey} as a particular case. That is,
\begin{equation}
D_{1,0}(n+\ell;1):=B_{n+\ell}=\sum_{j=0}^{\ell}\sum_{k=0}^n\sstirling{\ell}{j}\binom{n}{k}j^{n-k}B_k.
\end{equation}
Hence, \eqref{mainresult} is a generalization of Spivey's Bell number formula. Using \eqref{rdowlingdef} to both sides of \eqref{mainresult} yields
\begin{equation}
\sum_{i=0}^{n+\ell}W_{m,r}(n+\ell,i)x^i=\sum_{j=0}^{\ell}\sum_{k=0}^nW_{m,r}(\ell,j)\binom{n}{k}(mj)^{n-k}\sum_{i=j}^{k}W_{m,r}(k,i-j)x^i.
\end{equation}
Since $0\leq k\leq n+\ell$, then by comparing the coefficients of $x^i$, we see that
\begin{equation}
W_{m,r}(n+\ell,i)=\sum_{j=0}^{\ell}\sum_{k=0}^nW_{m,r}(\ell,j)\binom{n}{k}(mj)^{n-k}W_{m,r}(k,i-j).\label{newrec1}
\end{equation}
When $\ell=1$ in \eqref{mainresult} and \eqref{newrec1}, we consequently obtain
\begin{equation}
D_{m,r}(n+1;x)=\sum_{k=0}^n\binom{n}{k}m^{n-k}D_{m,r}(k;x)\label{maincor1}
\end{equation}
and
\begin{equation}
W_{m,r}(n+1,i)=\sum_{k=0}^n\binom{n}{k}m^{n-k}W_{m,r}(k,i-j).\label{newrec2}
\end{equation}
Furthermore, we get
\begin{equation}
D_{m,r}(n+1)=\sum_{k=0}^n\binom{n}{k}m^{n-k}D_{m,r}(k)\label{maincor2}
\end{equation}
by setting $x=1$ in \eqref{maincor1}. On the other hand, when $n=0$ in \eqref{mainresult}, we obtain the defining relation in \eqref{rdowlingdef}. That is,
\begin{equation}
D_{m,r}(\ell;x)=\sum_{j=0}^{\ell}W_{m,r}(\ell,j)x^j.
\end{equation}
This yields 
\begin{equation}
D_{m,r}(\ell)=\sum_{j=0}^{\ell}W_{m,r}(\ell,j)\label{maincor3}
\end{equation}
when $x=1$. Equations \eqref{newrec1} and \eqref{newrec2} are new recurrence relations for the $r$-Whitney numbers of the second kind. Also, we see that \eqref{mainresult} unifies \eqref{maincor2} and \eqref{maincor3} in the same way that Spivey's Bell number formula unifies \eqref{Belldef} and \eqref{Bellrec}.

\section{Other Generalizations}\label{sec4}

In the paper of Xu \cite{Xu}, the following result was obtained:
\begin{equation}
B_{n+m;\alpha,\beta,r}(x)=\sum_{k=0}^n\sum_{j=0}^m\binom{n}{k}x^jB_{k;\alpha,\beta,r}(x)S(n,k;\alpha,\beta,r)\prod_{i=0}^{n-k-1}(j\beta-(m+i)\alpha).\label{xuresult}
\end{equation}
Here, $B_{n+m;\alpha,\beta,r}(x)$ denote the \emph{generalized Bell polynomials} defined by
\begin{equation}
B_{n;\alpha,\beta,r}(x)=\sum_{i=0}^nS(n,i;\alpha,\beta,r)x^i,
\end{equation}
where $S(n,i;\alpha,\beta,r)$ denote the \emph{generalized Stirling numbers} (see \cite{Hsu}). It can be verified that the result in \eqref{mainresult} is the special case of \eqref{xuresult} when $\alpha=0$ and $\beta=m$. However, the approach used by Xu \cite{Xu} in deriving \eqref{xuresult} is motivated by the work of Gould and Quaintance \cite{Gould}, a method that is different from ours. 

In this section, more generalizations of Spivey's Bell number formula in terms of other generalizations of the Bell numbers are presented and discussed in addition to \eqref{xuresult}. These generalizations are stated in Corollaries \ref{cor1} to \ref{cor4} below which follow directly from Theorem \ref{maintheorem}.
\begin{corollary}\label{cor1}
For non-negative integers $n$ and $\ell$, and real number $r$,
\begin{equation}
B_{n+\ell,r}(x)=\sum_{j=0}^{\ell}\sum_{k=0}^n\sstirling{\ell+r}{j+r}_r\binom{n}{k}j^{n-k}B_k(x)x^j
\end{equation}
and
\begin{equation}
B_{n+\ell,r}=\sum_{j=0}^{\ell}\sum_{k=0}^n\sstirling{\ell+r}{j+r}_r\binom{n}{k}j^{n-k}B_k,
\end{equation}
where $B_{n+\ell,r}(x)$ and $B_{n+\ell,r}$ denote the \emph{$r$-Bell polynomials and numbers \cite{Mez2}} defined by
\begin{equation}
B_{n,r}(x)=\sum_{k=0}^n\sstirling{n+r}{k+r}_rx^k
\end{equation}
and $B_{n,r}:=B_{n,r}(1)$, respectively.
\end{corollary}
\begin{proof}
The proof is done by setting $m=1$ in \eqref{mainresult} and \eqref{mainresult2}, and then using \eqref{x1}.
\end{proof}

\begin{corollary}\label{cor2}
For non-negative integers $n$ and $\ell$, and real number $\alpha$,
\begin{equation}
\bar{B}_{\alpha}(n+\ell;x)=\sum_{j=0}^{\ell}\sum_{k=0}^nS_{\alpha}(\ell,j)\binom{n}{k}j^{n-k}\bar{B}_{\alpha}(k;x)x^j
\end{equation}
and
\begin{equation}
\bar{B}_{\alpha}(n+\ell)=\sum_{j=0}^{\ell}\sum_{k=0}^nS_{\alpha}(\ell,j)\binom{n}{k}j^{n-k}\bar{B}_{\alpha}(k),
\end{equation}
where $\bar{B}_{\alpha}(n+\ell;x)$ denote the ``polynomial'' extension of the \emph{noncentral Bell numbers $B_{\alpha}(n)$ \cite{Corcino2}} defined by
\begin{equation}
B_{\alpha}(n)=\sum_{k=0}^nS_{\alpha}(n,k).
\end{equation}
\end{corollary}
\begin{proof}
The proof is done by setting $m=1$ and $r=-\alpha$ in \eqref{mainresult} and \eqref{mainresult2}, and then using \eqref{x3}.
\end{proof}

\begin{corollary}\label{cor3}
For non-negative integers $n$ and $\ell$, and real number $\alpha$,
\begin{equation}
\widetilde{D}_{(\alpha)}(n+\ell;x)=\sum_{j=0}^{\ell}\sum_{k=0}^n\widetilde{W}_{(\alpha)}(\ell,j)\binom{n}{k}(j\alpha)^{n-k}\widetilde{D}_{(\alpha)}(k;x)x^j
\end{equation}
and
\begin{equation}
\widetilde{D}_{(\alpha)}(n+\ell;x)=\sum_{j=0}^{\ell}\sum_{k=0}^n\widetilde{W}_{(\alpha)}(\ell,j)\binom{n}{k}(j\alpha)^{n-k}\widetilde{D}_{(\alpha)}(k;x),
\end{equation}
where $\widetilde{D}_{(\alpha)}(n+\ell;x)$ and $\widetilde{D}_{(\alpha)}(n+\ell)$ denote the \emph{translated Dowling polynomials and numbers \cite{Mah1}} defined by
\begin{equation}
\widetilde{D}_{(\alpha)}(n;x)=\sum_{k=0}^nW_{(\alpha)}(n,k)x^k
\end{equation}
and
\begin{equation}
\widetilde{D}_{(\alpha)}(n)=\sum_{k=0}^nW_{(\alpha)}(n,k),
\end{equation}
respectively.
\end{corollary}
\begin{proof}
The proof is done by setting $m=\alpha$ and $r=0$ in \eqref{mainresult} and \eqref{mainresult2}, and then using \eqref{x4}.
\end{proof}

\begin{corollary}\label{cor4}
For non-negative integers $n$ and $\ell$, and real number $\alpha$,
\begin{equation}
D_m(n+\ell;x)=\sum_{j=0}^{\ell}\sum_{k=0}^nW_m(\ell,j)\binom{n}{k}j^{n-k}D_m(k;x)x^j
\end{equation}
and
\begin{equation}
D_m(n+\ell)=\sum_{j=0}^{\ell}\sum_{k=0}^nW_m(\ell,j)\binom{n}{k}j^{n-k}D_m(k),
\end{equation}
where $D_m(n+\ell;x)$ denote the ``polynomial'' extension of the \emph{Dowling numbers $D_m(n+\ell)$ \cite{Benoumhani}} defined by
\begin{equation}
D_m(n)=\sum_{k=0}^nW_m(n,k).
\end{equation}
\end{corollary}
\begin{proof}
The proof is done by setting $r=0$ in \eqref{mainresult} and \eqref{mainresult2}, and then using \eqref{x2}.
\end{proof}

One generalization of \eqref{spivey} is the following which can be seen in the paper of \cite{Mansour2}:
\begin{equation}
\b_{a;b}(n+m)=\sum_{k=0}^n\sum_{j=0}^m\binom{n}{k}\s_{a;b}(m,j)\left\{\prod_{i=0}^{n-k-1}(bj+a(i+m))\right\}\b_{a;b}(k).\label{mansourresult1}
\end{equation}
Here, $\b_{a;b}(n+m)$ denote a certain generalization of the Bell numbers defined as the sum of $\s_{a;b}(m,j)$, a certain generalization of the Stirling numbers of the second kind with recurrence relation given by
\begin{equation}
\s_{s;h}(n+1,k)=\s_{s;h}(n,k-1)+h\left(k+s(n-k)\right)\s_{s;h}(n,k)
\end{equation}
(see \cite[Proposition\ 3.2]{Mansour2}). If we compare this with the recurrence relation \cite[Equation\ 7]{Hsu}
\begin{equation}
S(n+1,k;\alpha,\beta,r)=S(n,k-1;\alpha,\beta,r)+(k\beta-n\alpha+r)S(n,k;\alpha,\beta,r)
\end{equation}
for the generalized Stirling numbers and with \eqref{recrWhitney}, we see that $\s_{s;h}(n,k)$ is not a generalization of the $r$-Whitney numbers of the second kind but a special case of $S(n,k;\alpha,\beta,r)$ when $\alpha=-hs$, $\beta=h(1-s)$ and $r=0$. This means that our results in Theorem \ref{maintheorem} are not generalized by \eqref{mansourresult1}.

Lastly, another generalization of \eqref{spivey} was due to Corcino et al. \cite{CorcinoCelesteGonzales}. Their result is
\begin{equation}
B_{s,q}[n+m;x]=\sum_{r=0}^n\sum_{j=0}^mS_{s,q}[m,j]q^{r(j(1-s)+sm)}\binom{n}{r}_{q^s}B_{s,q}[r;x]x^j\prod_{i=0}^{n-r-1}[j(1-s)+sm+si]_q,\label{Corcinoresult}
\end{equation}
where 
\begin{equation}
B_{s,q}[n;x]=\sum_{k=0}^nS_{s,q}[n,k]x^k
\end{equation}
is a certain $q$-deformed generalization of the Bell polynomials and $S_{s,q}[n,k]$ is a certain $q$-deformed generalization of the Stirling numbers of the second kind with recurrence relation given by
\begin{equation}
S_{s,q}[n,k]x^k=q^{s(n-1)-(s-1)(k-1)}S_{s,q}[n-1,k-1]x^k+[s(n-1)-(s-1)k]+qS_{s,q}[n-1,k]x^k
\end{equation}
(see \cite[Proposition\ 1]{CorcinoCelesteGonzales}). It is clear to see that the limit as $q\rightarrow1$ of this recurrence relation yields a certain number which is a particular case of the generalized Stirling numbers $S(n,k;\alpha,\beta,r)$ when $\alpha=-s$, $\beta=1-s$ and $r=-s$. Still, the said case is not a generalization of the $r$-Whitney numbers of the second kind which also means that \eqref{Corcinoresult} does not generalize our results in Theorem \ref{maintheorem}. However, it was remarked in their paper that Katriel's \cite{Katriel} result in \eqref{katrielresult} can be obtained from \eqref{Corcinoresult} by setting $s=0$ (see \cite[Remark\ 4]{CorcinoCelesteGonzales}).

\section*{Acknowledgment}
The authors are thankful to the referees for giving comments and suggestions which greatly improved the paper. The authors dedicate this paper to their families and all other victims of the war in Marawi City last May 23, 2017 to October 17, 2017.

\end{document}